\documentclass{amsart}

\usepackage{amsmath, amsthm, amssymb} % standard packages
\usepackage[colorlinks=true, allcolors=blue]{hyperref} % linking
\usepackage[capitalize, nameinlink, noabbrev]{cleveref} % referencing
\usepackage{asymptote} % diagrams
\usepackage{dsfont} % for \mathds{1}
\usepackage{mathtools} % for \mathclap

\newtheorem{theorem}{Theorem}[section]
\newtheorem{lemma}[theorem]{Lemma}
\newtheorem{claim}[theorem]{Claim}
\newtheorem{corollary}[theorem]{Corollary}
\newtheorem{fact}[theorem]{Fact}

\allowdisplaybreaks % so align environments don't create extra space

\begin{asydef}
    texpreamble("\usepackage{amsmath}");

    bool contains(int[] arr, int value) {
        for (int i = 0; i < arr.length; ++i) {
            if (arr[i] == value) {
                return true;
            }
        }
        return false;
    }
    
    // cycles through the rainbow as r varies from 0 to 1
    pen rainbow(real r, real fading = 0) {
        r = 6 * ((r+1/12) % 1);
        
        real s = r % 1;
        pen c;
        if (r < 1) {
            c = rgb(1, s, 0);
        } else if (r < 2) {
            c = rgb(1-s, 1, 0);
        } else if (r < 3) {
            c = rgb(0, 1, s);
        } else if (r < 4) {
            c = rgb(0, 1-s, 1);
        } else if (r < 5) {
            c = rgb(s, 0, 1);
        } else {
            c = rgb(1, 0, 1-s);
        }
        return interp(c, white, fading);
    }

    string ordinal(int n) {
        if (n == 1) {
            return "$1^\text{st}$";
        } else if (n == 2) {
            return "$2^\text{nd}$";
        } else if (n == 3) {
            return "$3^\text{rd}$";
        } else if (n == 4) {
            return "$4^\text{th}$";
        } else if (n == 5) {
            return "$5^\text{th}$";
        } else {
            return "?";
        }
    }

    // pictoral representation of a cycle
    // p = cycle, c = center, s = radius
    void drawcycle(int[] p, int[] colors, pair c = (0, 0), real s = 1, bool card = false, real labelsize=1, bool colorboundary = false, int[] special = new int[] {}) {
        int n = p.length;
        int r = 6;
        for (int i = 0; i < n; ++i) {
            draw(arc(c, s, i*360/n+7/s, (i+1)*360/n-7/s), ArcArrow);
            pair P = c+s*dir(i*360/n);
            if (contains(special, p[i])) {
                fill(circle(P, 0.12), black);
            }
            fill(circle(P, 1/10), rainbow(colors[i]/r, 2/3));
            if (colorboundary) {
                draw(circle(P, 1/10), rainbow(colors[i]/r));
            } else {
                draw(circle(P, 1/10));
                
            }
            if (card) {
                label(scale(labelsize)*ordinal(p[i]), P);
            } else {
                label(scale(labelsize)*string(p[i]), P);
            }
        }
    }

    // vertex of the six-vertex graph
    pair vertex(int n) {
        return dir(360*n/6);
    }

    void drawvertices() {
        for (int i = 0; i < 6; ++i) {
            fill(circle(vertex(i), 0.1), rainbow(i/6, 1/3));
        }
    }
    
    // draws edges of the six-vertex graph
    path drawedge(int a, int b, pair da = dir(vertex(b)-vertex(a)), pair db = dir(vertex(a)-vertex(b)), bool isthick = false) {
        pair A = vertex(a);
        pair B = vertex(b);
        pair A1 = A + 0.18*da;
        pair B1 = B + 0.18*db;
        pair A2 = A + 0.3*da;
        pair B2 = B + 0.3*db;
        path p = A1 -- A2{da} .. tension 0.75 .. B2{db} -- B1;
        if (isthick) {
            draw(p, linewidth(2), ArcArrow(4bp));
        } else {
            draw(p, ArcArrow);
        }
        return p;
    }

    void edgelabel1(int a, int b, int l, pair da = dir(vertex(b)-vertex(a)), pair db = dir(vertex(a)-vertex(b))) {
        path p = drawedge(a, b, da, db);
        real along = 0.4;
        pair M = relpoint(p, along);
        real eps = 0.01;
        pair P = dir(relpoint(p, along+eps)-relpoint(p, along-eps));
        pair s = 0.075;
        label(scale(0.4)*string(l), M + 2/3*s*P*dir(90));
    }

    void edgelabel2(int a, int b, int l, pair da = dir(vertex(b)-vertex(a)), pair db = dir(vertex(a)-vertex(b))) {
        path p = drawedge(a, b, da, db);
        real along = 0.4;
        pair M = relpoint(p, along);
        real eps = 0.01;
        pair P = dir(relpoint(p, along+eps)-relpoint(p, along-eps));
        pair s = 0.075;
        label(scale(0.4)*string(l), M - s*P*dir(90));
        fill(circle(M - s*P*dir(90), 0.06), lightgray);
        draw(circle(M - s*P*dir(90), 0.06), gray);
    }
    
    // edge of the six-vertex graph
    void edgelabels(int a, int b, int l1, int l2, pair da = dir(vertex(b)-vertex(a)), pair db = dir(vertex(a)-vertex(b))) {
        edgelabel1(a, b, l1, da, db);
        edgelabel2(a, b, l2, da, db);
    }

    void connector(int a, pair d1, pair d2) {
        pair P = vertex(a);
        pair A1 = P + 0.1*d1;
        pair A2 = P + 0.16*d1;
        pair B1 = P + 0.1*d2;
        pair B2 = P + 0.16*d2;
        path p = A2 -- A1{-d1} .. tension 0.75 .. B1{-d2} -- B2;
        draw(p);
        draw(B2+1/25*d2*dir(135) -- B2 -- B2+1/25*d2*dir(-135));
    }

    void exitcounter(int a, int exit, pair d) {
        string l = ordinal(exit);
        pair C = vertex(a)+0.21*d;
        fill(circle(C, 0.05), rainbow(a/6, 2/3));
        draw(circle(C, 0.05), rainbow(a/6));
        label(scale(1/3)*l, C);
    }

    void numbercircle(int c, int l, pair P) {
        fill(circle(P, 1/10), rainbow(c/6, 2/3));
        draw(circle(P, 1/10), rainbow(c/6));
        label(scale(2/3)*ordinal(l), P);
    }

    void rectangle(pair A, pair B) {
        draw((A.x, A.y) -- (A.x, B.y) -- (B.x, B.y) -- (B.x, A.y) -- cycle);
    }
\end{asydef}

\title{Coalescence Probabilities of Cycle Products}
\author{Holden Mui}
\date{\today}

\begin{document}

\begin{abstract} 
    Generalizing a formula of Stanley, we prove combinatorially that the probability that $1, 2, \dots, k$ are contained in the same cycle of a product of two random $n$-cycles is
    \[\frac{1}{k} + \frac{4 (-1)^n}{ \binom{2k}{k}} \sum_{\substack{1 \leq i \leq k-1 \\ i \not\equiv n \bmod 2}} \binom{2k-1}{k+i} \left(\frac{1}{n+i+1} - \frac{1}{n-i}\right).\]
\end{abstract}

\maketitle

\section{Introduction}
The study of cycle products originates in the 1980s with seminal papers from Boccara and Stanley. Boccara \cite{Boc80} used analytic techniques to obtain an integral formula for the number of ways to factor a permutation into two cycles. More generally, Stanley \cite{Sta81} used the character theory of the symmetric group to compute the number of ways to factor a permutation into any given number of cycles. These results sparked several decades of research into understanding permutation products via cycles.

One well-studied question in this domain asks about the cycle count of a product of two cycles. Kwak and Lee \cite{KL93} used a character-theoretic result of Jackson \cite{Jac88} to show that the probability that a product of two random $n$-cycles has $\nu$ cycles is
\[\begin{cases}
    \displaystyle 2 \Pr_{\sigma \in S_{n+1}}[c(\sigma) = \nu] & \nu \equiv n \bmod 2 \\
    0 & \nu \not\equiv n \bmod 2,
\end{cases}\]
where $c(\sigma)$ denotes $\sigma$'s cycle count. This formula was also obtained independently by Zagier \cite{Zag95} via the representation theory of the symmetric group and Stanley \cite{Sta11} via Schur functions, power sum symmetric functions, and character theory.

However, combinatorial proofs of the above results remained elusive until Cangelmi \cite{Can03} resolved the $\nu=1$ case. In particular, Cangelmi proved bijectively that for odd $n$, the probability that a product of two random $n$-cycles is an $n$-cycle is $\smash{\frac{2}{n+1}}$. The general $\nu$ case was resolved combinatorially by F\'eray and Vassilieva \cite{FV10}, and other combinatorial proofs were given in \cite{FV12} and \cite{CR16}.

In addition to studying the cycle count of a product of two cycles, one can also study coalescence and separation probabilities. B\'{o}na and Flynn \cite{BF09} conjectured that the probability that 1 and 2 are in the same cycle of a product of two random $n$-cycles is $\frac{1}{2}$ for odd $n$, and asked about the value for even $n$. Their conjecture was resolved when Stanley \cite{Sta11} showed that this probability is
\[
\begin{cases}
    \frac{1}{2} & \text{$n \geq 2$ odd} \\
    \frac{1}{2} - \frac{2}{(n-1)(n+2)} & \text{$n \geq 2$ even}
\end{cases}
\]
using analytic methods and a formula of Bocarra \cite{Boc80}.
Additionally, Stanley also showed that the probability 1, 2, and 3 are in the same cycle of a product of two random $n$-cycles is
\[
\begin{cases}
    \frac{1}{3} + \frac{1}{(n-2)(n+3)} & \text{$n \geq 3$ odd} \\
    \frac{1}{3} - \frac{3}{(n-1)(n+2)} & \text{$n \geq 3$ even}.
\end{cases}
\]
For the general case, B\'{o}na and Pittel \cite{BP21} use a character-based Fourier transform  to show that the probability that $1, 2, \dots, k$ are all contained in the same cycle of a product of two random $n$-cycles is
\[\frac{1}{k} - \frac{1}{n(n+1)} - \frac{(-1)^{k}}{\binom{n-1}{k-1}} \sum_{i=0}^{n-k} (-1)^i \binom{n-1}{i} \frac{1}{i+k+1}\]
for $k \leq n$. They also show that their expression is a rational function in $n(n+1)$ for $2 \leq n \leq 5$ but state that there is no discernible behavior otherwise.

On the other hand, Bernardi, Du, Morales, and Stanley \cite{BDMS14} studied separation probabilities; they showed that the probability that $1, 2, \dots, k$ are all contained in different cycles in a product of two random $n$-cycles is
\[\begin{cases}
    \frac{1}{k!} & n \not\equiv k \bmod 2 \\
    \frac{1}{k!} + \frac{2}{(k-2)! (n-k+1)(n+k)} & n \equiv k \bmod 2.
\end{cases}\]
Notably, their proof is combinatorial and does not rely on the machinery of the character-based Fourier transform. This begs the following question: can B\'ona and Pittel's result also be obtained combinatorially?

Inspired by Bernardi's \cite{Ber12} presentation of Lass' \cite{Lass01} proof of the Harer-Zagier formula, this paper uses combinatorial methods to obtain a formula for the probability that $1, 2, \ldots, k$ are all contained in the same cycle of a product of two random $n$-cycles. The formula is algebraically equivalent to Bon\'a and Pittel's result but is written so that the behavior for general $n$ is discernible. Similar to \cite{BDMS14} and \cite{SV08}, the proof uses a diagrammatic framework for obtaining the permutation statistics of cycle products combinatorially via a bijection to pairs of Eulerian circuits on graphs. An extension of this bijection is used to obtain the result of this paper:
\begin{theorem}\label{theorem:coalescence}
    Let $k$ and $n$ be positive integers with $k \leq n$. The probability that $1, 2, \dots, k$ are contained in the same cycle of a product of two random $n$-cycles is
    \[\frac{1}{k} + \frac{4 (-1)^n}{ \binom{2k}{k}} \sum_{\substack{1 \leq i \leq k-1 \\ i \not\equiv n \bmod 2}} \binom{2k-1}{k+i} \left(\frac{1}{n+i+1} - \frac{1}{n-i}\right).\]
\end{theorem}
Notably, this formula is written in a way that makes it clear that this coalescence probability is a rational function in $n(n+1)$ for each possible parity of $n$. Furthermore, it shows that B\'ona and Pittel's result indeed exhibits ``discernibly regular behavior,'' and it resolves a specific case of a conjecture by Stanley \cite[p.~53]{Sta10}. These expansions, for small values of $k$, are given in \cref{table:coalescence}.

\begin{table}[ht]
    \centering
    \renewcommand{\arraystretch}{1.5}
    \begin{tabular}{c|cc}
        & $n$ even & $n$ odd \\ \hline
        $k=1$ & $1$ & $1$ \\
        $k=2$ & $\frac{1}{2} - \frac{2/3}{n-1} + \frac{2/3}{n+2}$ & $\frac{1}{2}$ \\
        $k=3$ & $\frac{1}{3} - \frac{1}{n-1} + \frac{1}{n+2}$ & $\frac{1}{3} + \frac{1/5}{n-2} - \frac{1/5}{n+3}$ \\
        $k=4$ & $\frac{1}{4} - \frac{2/35}{n-3} - \frac{6/5}{n-1} + \frac{6/5}{n+2} + \frac{2/35}{n+4}$ & $\frac{1}{4} + \frac{2/5}{n-2} - \frac{2/5}{n+3}$ \\
        $k=5$ & $\frac{1}{5} - \frac{1/7}{n-3} - \frac{4/3}{n-1} + \frac{4/3}{n+2} + \frac{1/7}{n+4}$ & $\frac{1}{5} + \frac{1/63}{n-4} + \frac{4/7}{n-2} - \frac{4/7}{n+3} - \frac{1/63}{n+5}$
    \end{tabular}
    \vspace{\baselineskip}
    \caption{Coalescence probabilities (\cref{theorem:coalescence}).}
    \label{table:coalescence}
\end{table}

\subsection*{Outline}
\cref{section:setup} defines colored subsets of colored cycles and outlines the method of attack. \cref{section:bijection} gives a bijection between colored cycles and pairs of Eulerian tours on directed graphs, then extends the bijection to colored subsets of colored cycles. The main result, \cref{theorem:coalescence}, is proved in \cref{section:coalescence} via the bijection and extensive combinatorial manipulations on the resulting expression. Finally, future directions are given in \cref{section:future}. 

\section{The setup}\label{section:setup}
Fix a positive integer $n$. By relabeling, the cycle statistics of a product of $n$-cycles are identical to the cycle statistics of $\sigma \tau$, where $\sigma$ is a random $n$-cycle and $\tau = (n \dots 21)$. This motivates the definition of an \emph{$r$-colored $n$-cycle}, which consists of two pieces of data:
\begin{itemize}
    \item an $n$-cycle $\sigma$, and
    \item a coloring of $\{1, 2, \dots, n\}$ using colors indexed by $\{1, 2, \dots, r\}$.
\end{itemize}
The $n$-cycle $\sigma$ and the coloring must satisfy the following conditions:
\begin{itemize}
    \item every color must be used at least once, and
    \item every cycle in $\sigma \tau$ must be monochromatic.
\end{itemize}
An example of a 6-colored 16-cycle is depicted in \cref{figure:coloredcycle}.

\begin{figure}[ht]
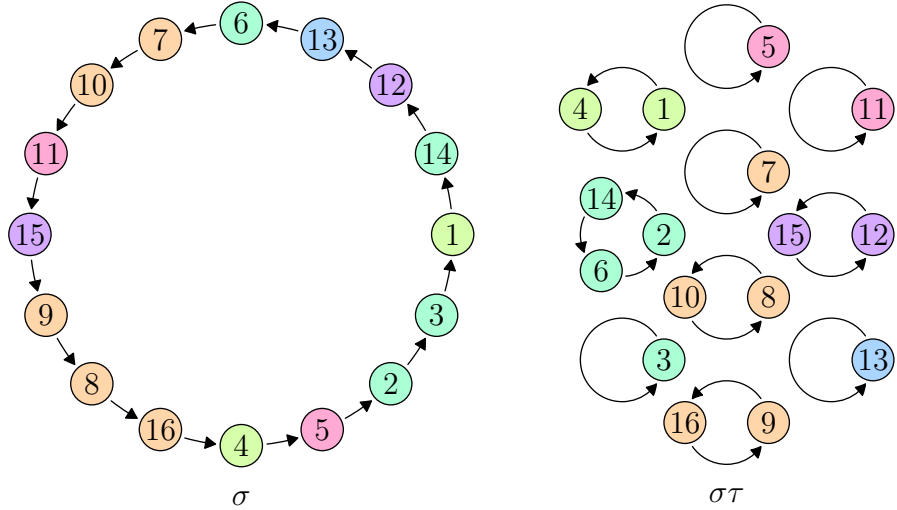

    \centering
    \begin{minipage}{0.49\textwidth}
        \centering
        \begin{asy}[width=\textwidth]
            int p[] = {1, 14, 12, 13, 6, 7, 10, 11, 15, 9, 8, 16, 4, 5, 2, 3};
            int colors[] = {1, 2, 4, 3, 2, 0, 0, 5, 4, 0, 0, 0, 1, 5, 2, 2};
            drawcycle(p, colors);

            label("$\sigma$", (0, -5/4));
        \end{asy}
    \end{minipage} \hfill
    \begin{minipage}{0.49\textwidth}
        \centering
        \begin{asy}[width=0.72\textwidth]
            pair[] locations = {(-1/2, 3/5), (-1/2, 0), (-1/2, -3/5), (0, 9/10), (0, 3/10), (0, -3/10), (0, -9/10), (1/2, 3/5), (1/2, 0), (1/2, -3/5)};
            
            drawcycle(new int[] {1, 4}, new int[] {1, 1}, locations[0], 1/5);
            drawcycle(new int[] {2, 14, 6}, new int[] {2, 2, 2}, locations[1], 1/5);
            drawcycle(new int[] {3}, new int[] {2}, locations[2], 1/5);
            drawcycle(new int[] {5}, new int[] {5}, locations[3], 1/5);
            drawcycle(new int[] {7}, new int[] {0}, locations[4], 1/5);
            drawcycle(new int[] {8, 10}, new int[] {0, 0}, locations[5], 1/5);
            drawcycle(new int[] {9, 16}, new int[] {0, 0}, locations[6], 1/5);
            drawcycle(new int[] {11}, new int[] {5}, locations[7], 1/5);
            drawcycle(new int[] {12, 15}, new int[] {4, 4}, locations[8], 1/5);
            drawcycle(new int[] {13}, new int[] {3}, locations[9], 1/5);
            label("$\sigma\tau$", (0, -5/4));
        \end{asy}
    \end{minipage}
    \caption{A 6-colored 16-cycle $\sigma$ and the colored cycles of $\sigma \tau$.}
    \label{figure:coloredcycle}
\end{figure}

The $r$-colored $n$-cycle idea can be generalized to a \emph{$t$-colored $k$-subset of an $r$-colored $n$-cycle}, which consists of an $r$-colored $n$-cycle and a $k$-element subset of $\{1, 2, \dots, n\}$ such that the subset has exactly $t$ distinct colors. An example of a 6-colored 16-cycle with a 3-colored 5-subset is depicted in \cref{figure:coloredcyclesubset}.

\begin{figure}[ht]
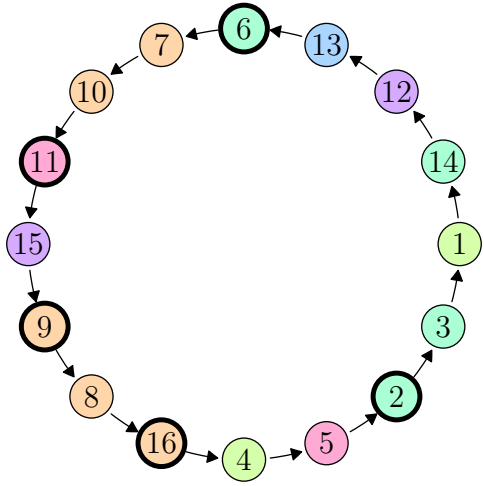

    \centering
    \begin{asy}[width=0.5\textwidth]
        int p[] = {1, 14, 12, 13, 6, 7, 10, 11, 15, 9, 8, 16, 4, 5, 2, 3};
        int colors[] = {1, 2, 4, 3, 2, 0, 0, 5, 4, 0, 0, 0, 1, 5, 2, 2};
        drawcycle(p, colors, special = new int[] {2, 6, 9, 11, 16});
    \end{asy}
    \caption{A 3-colored 5-subset of a 6-colored 16-cycle.}
    \label{figure:coloredcyclesubset}
\end{figure}

Counting $t$-colored $k$-subsets of $r$-colored $n$-cycles enables the computation of coalescence probabilities:

\begin{lemma}\label{lemma:coalescence}
    Let $k$ and $n$ be positive integers such that $k \leq n$. The probability that $1, 2, \dots, k$ are all contained in the same cycle of a product of two random $n$-cycles is
    \[\frac{(-1)^n}{(n-1)! \binom{n}{k}}\sum_{t=1}^k \frac{1}{t} \sum_{r=1}^n (-1)^r \#\left\{\raisebox{-.25\height}{\text{\shortstack{$t$-colored $k$-subsets\\of $r$-colored $n$-cycles}}}\right\}.\]
\end{lemma}

To prove this lemma, one first needs the following two claims:

\begin{claim}\label{claim:partitionsum}
    Let $b$ be a positive integer, and let $S(b, t)$ denote the number of equivalence relations on $\{1, 2, \dots, b\}$ with $t$ equivalence classes. Then
    \[\sum_{t=1}^b (-1)^{t-1}(t-1)! S(b, t) = \begin{cases}
        1 & b=1 \\
        0 & \text{otherwise}.
    \end{cases}\]
\end{claim}

\begin{proof}
    The identity
    \[\sum_{t=1}^b S(b, t) x(x-1)(x-2) \cdots (x-(t-1)) = x^b,\]
    follows by counting colorings of $\{1, 2, \dots, b\}$ using $x$ colors by number of distinct colors used. Dividing both sides by $x$ and substituting $x=0$ gives the result.
\end{proof}

\begin{claim}\label{claim:cycleparity}
    Let $\sigma$ and $\tau$ be $n$-cycles. Then
    $c(\sigma \tau) \equiv n \bmod{2}$, where $c(\sigma \tau)$ denotes the cycle count of $\sigma \tau$.
\end{claim}

\begin{proof}
    If $\sigma \tau$ has cycle lengths $\ell_1, \dots, \ell_{c(\sigma \tau)}$, then
    \[1 = \operatorname{sgn}(\sigma \tau) = (-1)^{\ell_1-1} \cdots (-1)^{\ell_{c(\sigma \tau)}-1} = (-1)^{\ell_1 + \cdots + \ell_{c(\sigma \tau)} - c(\sigma \tau)} = (-1)^{n - c(\sigma \tau)}.\]
    Therefore $c(\sigma \tau) \equiv n \bmod{2}$.
\end{proof}

Now, \cref{lemma:coalescence} can be proved:

\begin{proof}[Proof of \cref{lemma:coalescence}]
    Define a \emph{$b$-cycle $k$-subset} of a permutation $\pi \in S_n$ to be a $k$-element subset of $\{1, 2, \dots, n\}$ contained in exactly $b$ distinct cycles of $\pi$. Then for any integer $t$ between 1 and $k$,
    \[\sum_{\substack{\sigma \in S_n \\ c(\sigma)=1}} \sum_{b=t}^k \#\left\{\raisebox{-.25\height}{\text{\shortstack{$b$-cycle $k$-\\subsets of $\sigma \tau$}}}\right\} \#\left\{\raisebox{-.25\height}{\text{\shortstack{equivalence relations on \\ $\{1, 2, \dots, b\}$ with $t$ classes}}}\right\} q^{c(\sigma\tau)-b} \prod_{i=0}^{t-1} (q - i)\]
    counts the number of ways to 
    \begin{itemize}
        \item pick an $n$-cycle $\sigma$ and an integer $b$ between $t$ and $k$, inclusive, 
        \item select a $k$-element subset of $\{1, 2, \dots, n\}$ contained in exactly $b$ distinct cycles of $\sigma \tau$,
        \item partition the $b$ cycles into exactly $t$ equivalence classes,
        \item color the $c(\sigma\tau)-b$ other cycles using $q$ colors, and
        \item color the $b$ cycles using $q$ colors such that cycles are in the same color if and only if they are in the same equivalence classes.
    \end{itemize}
    This is equal to the number of ways to
    \begin{itemize}
        \item pick an $n$-cycle $\sigma$ and an integer $r$ between $1$ and $n$, inclusive,
        \item color the cycles of $\sigma \tau$ using exactly $r$ colors out of $q$ total colors, and
        \item pick a $k$-element subset of $\{1, 2, \dots, n\}$ colored with exactly $t$ colors,
    \end{itemize}
    which is
    \[\sum_{r=1}^n \binom{q}{r} \#\left\{\raisebox{-.25\height}{\text{\shortstack{$t$-colored $k$-subsets\\of $r$-colored $n$-cycles}}}\right\}.\]
    Rearranging this equality gives
    \[\frac{\displaystyle\sum_{r=1}^n \binom{q}{r} \#\left\{\raisebox{-.25\height}{\text{\shortstack{$t$-colored $k$-subsets\\of $r$-colored $n$-cycles}}}\right\}}{q(q-1)(q-2) \cdots (q-(t-1))} = \sum_{\substack{\sigma \in S_n \\ c(\sigma)=1}} \sum_{b=t}^k \#\left\{\raisebox{-.25\height}{\text{\shortstack{$b$-cycle $k$-\\subsets of $\sigma \tau$}}}\right\} S(b, t) q^{c(\sigma\tau)-b},\]
    where $S(b, t)$ denotes the number of equivalence relations on a $\{1, 2, \dots, b\}$ with $t$ classes. Summing over all $t$ from 1 to $k$, weighted by $(-1)^{t-1} (t-1)!q^{1-n}$, gives
    \begin{align*}
        &\phantom{{}={}}\sum_{t=1}^k \frac{(-1)^{t-1} (t-1)!q^{1-n}}{q(q-1)(q-2) \cdots (q-(t-1))} \sum_{r=1}^n \binom{q}{r} \#\left\{\raisebox{-.25\height}{\text{\shortstack{$t$-colored $k$-subsets\\of $r$-colored $n$-cycles}}}\right\} \\
        &= \sum_{t=1}^k (-1)^{t-1} (t-1)! q^{1-n} \sum_{\substack{\sigma \in S_n \\ c(\sigma)=1}} \sum_{b=t}^k \#\left\{\raisebox{-.25\height}{\text{\shortstack{$b$-cycle $k$-\\subsets of $\sigma \tau$}}}\right\} S(b, t) q^{c(\sigma\tau)-b} \\
        &= \sum_{\substack{\sigma \in S_n \\ c(\sigma)=1}} \sum_{b=1}^k \#\left\{\raisebox{-.25\height}{\text{\shortstack{$b$-cycle $k$-\\subsets of $\sigma \tau$}}}\right\} q^{c(\sigma\tau)-b+1-n} \sum_{t=1}^b (-1)^{t-1} (t-1)! S(b, t) \\
        &\overset{\mathclap{(\ref{claim:partitionsum})}}{=} \sum_{\substack{\sigma \in S_n \\ c(\sigma)=1}} \sum_{b=1}^k \#\left\{\raisebox{-.25\height}{\text{\shortstack{$b$-cycle $k$-\\subsets of $\sigma \tau$}}}\right\} q^{c(\sigma\tau)-b+1-n} \mathds{1}_{b=1} \\
        &= \sum_{\substack{\sigma \in S_n \\ c(\sigma)=1}} \#\left\{\raisebox{-.25\height}{\text{\shortstack{$1$-cycle $k$-\\subsets of $\sigma \tau$}}}\right\}q^{c(\sigma\tau)-n}.
    \end{align*}
    Finally, $c(\sigma \tau)$ always has the same parity as $n$ by \cref{claim:cycleparity}. Thus substituting $q=-1$ above gives
    \begin{align*}
        (-1)^n\sum_{t=1}^k \frac{1}{t} \sum_{r=1}^n (-1)^r \#\left\{\raisebox{-.25\height}{\text{\shortstack{$t$-colored $k$-subsets\\of $r$-colored $n$-cycles}}}\right\} = \#\left\{\raisebox{-.25\height}{\text{\shortstack{$n$-cycles $\sigma$ and $1$-cycle \\ $k$-subsets of $\sigma \tau$}}}\right\}.
    \end{align*}
    Since there are $(n-1)!\binom{n}{k}$ ways to choose an $n$-cycle and a $k$-element subset of $\{1, 2, \dots, n\}$, dividing both sides by $(n-1)! \binom{n}{k}$ gives the desired probability.
\end{proof}

\cref{lemma:coalescence} provides the method of attack for calculating coalescence probabilities. It remains to count $t$-colored $k$-subsets of $r$-colored $n$-cycles and evaluate the resulting sum to obtain the main result.

\section{The bijection}\label{section:bijection}

Given a positive integer $n$ and a sequence $s_1, \dots, s_r$ of positive integers summing to $n$, define an \emph{$(s_1, \dots, s_r)$-colored $n$-cycle} to be an $r$-colored $n$-cycle with $s_i$ colors of color $i$ for all $1 \leq i \leq r$. For example, the 6-colored 16-cycle in \cref{figure:coloredcycle} is a $(5, 2, 4, 1, 2, 2)$-colored 16-cycle, where the colors are ordered orange, lime green, sea green, blue, purple, and pink. This section will culminate in a formula for the number of $t$-colored $k$-subsets of $r$-colored $n$-cycles by
\begin{itemize}
    \item describing a bijection involving $(s_1, \dots, s_r)$-colored $n$-cycles (\cref{theorem:bijection}), which takes three subparts (\cref{lemma:bijectionstep1}, \cref{lemma:bijectionstep2}, and \cref{lemma:bijectionstep3}) to prove,
    \item using the bijection to count $r$-colored $n$-cycles (\cref{corollary:coloredcycle}), and
    \item extending the bijection to count $t$-colored $k$-subsets of $r$-colored $n$-cycles (\cref{corollary:coloredcyclesubset}).
\end{itemize}

\begin{theorem}\label{theorem:bijection}
    Let $n$ be a positive integer, let $s_1, \dots, s_r$ be a sequence of positive integers summing to $n$, and let $S$ be the $n$-element set
    \[S = \{(1, 1), (1, 2), \dots, (1, s_1), (2, 1), (2, 2), \dots, (2, s_2), \dots, (r, 1), (r, 2), \dots, (r, s_r)\}.\]
    There is a bijection between $(s_1, \dots, s_r)$-colored $n$-cycles and ways to arrange the elements of $S$ to form an $(r-1)$-term sequence and an $(n-r+1)$-element cycle.
\end{theorem}

\begin{figure}[ht]
    \centering
    \begin{minipage}{0.45\textwidth}
        \centering
        \begin{asy}[width=\textwidth]
            int p[] = {1, 14, 12, 13, 6, 7, 10, 11, 15, 9, 8, 16, 4, 5, 2, 3};
            int colors[] = {1, 2, 4, 3, 2, 0, 0, 5, 4, 0, 0, 0, 1, 5, 2, 2};
            drawcycle(p, colors, labelsize=4/5);

            rectangle((1.2, 1.2), (-1.2, -1.2));
        \end{asy}
    \end{minipage}
    \begin{minipage}{0.05\textwidth}
        \[\Longleftrightarrow\]
    \end{minipage}
    \begin{minipage}{0.45\textwidth}
        \centering
        \begin{asy}[width=\textwidth]
            numbercircle(2, 1, (-0.6, 0));
            numbercircle(0, 5, (-0.3, 0));
            numbercircle(1, 2, (0, 0));
            numbercircle(2, 2, (0.3, 0));
            numbercircle(1, 1, (0.6, 0));

            drawcycle(new int[] {3, 4, 2, 1, 4, 3, 2, 1, 1, 1, 2}, new int[] {2, 2, 5, 0, 0, 0, 0, 5, 4, 3, 4}, c =(0, -7/6), s=2/3, card=true, labelsize=2/3, colorboundary=true);

            label("$+$", (0, -0.25));

            rectangle((-1.1, -2), (1.1, 0.2));
        \end{asy}
    \end{minipage}
    \caption{The bijection in \cref{theorem:bijection}.}
    \label{figure:bijection}
\end{figure}

An example of this bijection is given in \cref{figure:bijection}. The elements of $S$ are depicted as colored circles containing ordinal numerals; the first coordinate determines the color and the second coordinate determines the ordinal numeral. 

The proof of \cref{theorem:bijection} has three subparts. To prove it, first define a \emph{degree-$(s_1, \dots, s_r)$ digraph}, where $s_1, \dots, s_r$ are positive integers summing to $n$, to be a directed graph with vertices $\{1, 2, \dots, r\}$ and $n$ directed edges such that vertex $i$ has indegree and outdegree $s_i$ for each $1 \leq i \leq r$. An example of a degree-$(5, 2, 4, 1, 2, 2)$ digraph is given in \cref{figure:digraph}.

\begin{figure}[h]
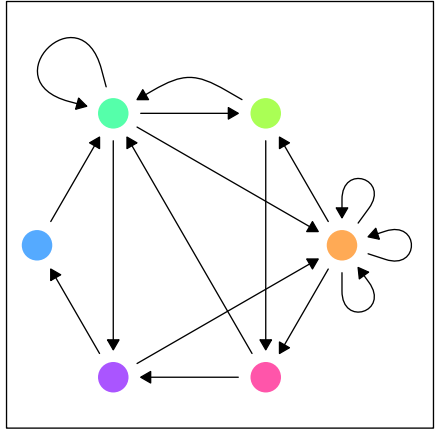

    \centering
    \begin{asy}[width=0.45\textwidth]
        drawvertices();

        drawedge(1, 2, dir(150), dir(30));
        drawedge(2, 4);
        drawedge(4, 3);
        drawedge(3, 2);
        drawedge(2, 0);
        drawedge(0, 0, dir(-90), dir(-54));
        drawedge(0, 5);
        drawedge(5, 4);
        drawedge(4, 0);
        drawedge(0, 0, dir(-18), dir(18));
        drawedge(0, 0, dir(54), dir(90));
        drawedge(0, 1);
        drawedge(1, 5);
        drawedge(5, 2);
        drawedge(2, 2, dir(105), dir(165));
        drawedge(2, 1);

        rectangle((1.6, 1.6), (-1.2, -1.2));
    \end{asy}
    \caption{A degree-$(5, 2, 4, 1, 2, 2)$ digraph.}
    \label{figure:digraph}
\end{figure}

The first step of the bijection, depicted in \cref{figure:bijectionstep1}, is:

\begin{lemma}\label{lemma:bijectionstep1}
    Let $s_1, \dots, s_r$ be a sequence of positive integers summing to $n$. There is a bijection between $(s_1, \dots, s_r)$-colored $n$-cycles and ordered pairs of Eulerian tours with the same starting edge on degree-$(s_1, \dots, s_r)$ digraphs.
\end{lemma}

\begin{figure}[p]
    \centering
    \begin{asy}[width=0.45\textwidth]
        int p[] = {1, 14, 12, 13, 6, 7, 10, 11, 15, 9, 8, 16, 4, 5, 2, 3};
        int colors[] = {1, 2, 4, 3, 2, 0, 0, 5, 4, 0, 0, 0, 1, 5, 2, 2};
        drawcycle(p, colors, labelsize=4/5);

        rectangle((1.2, 1.2), (-1.2, -1.2));
    \end{asy}
    \\ $\Downarrow \text{\scriptsize{(1a)}}$ \\
    \begin{asy}[width=0.45\textwidth]
        int p[] = {1, 14, 12, 13, 6, 7, 10, 11, 15, 9, 8, 16, 4, 5, 2, 3};
        int colors[] = {1, 2, 4, 3, 2, 0, 0, 5, 4, 0, 0, 0, 1, 5, 2, 2};
        
        int n = p.length;
        int r = 6;
        for (int i = 0; i < n; ++i) {
            draw(arc((0, 0), 1, i*360/n+7, (i+1)*360/n-7), ArcArrow);
            pair P = dir(i*360/n);
            fill(circle(P, 1/10), rainbow(colors[i]/r, 1/3));
            fill(circle(1.075*P*dir(0.4*360/n), 0.05), lightgray);
            draw(circle(1.075*P*dir(0.4*360/n), 0.05), gray);
            label(scale(0.5)*string(i+1), 0.95*P*dir(0.4*360/n));
            label(scale(0.5)*string(p[i]), 1.075*P*dir(0.4*360/n));
        }

        rectangle((1.2, 1.2), (-1.2, -1.2));
    \end{asy}
    \\ $\Downarrow \text{\scriptsize{(1b)}}$ \\
    \begin{asy}[width=0.45\textwidth]
        drawvertices();

        edgelabels(1, 2, 1, 1, dir(150), dir(30));
        edgelabels(2, 4, 2, 14);
        edgelabels(4, 3, 3, 12);
        edgelabels(3, 2, 4, 13);
        edgelabels(2, 0, 5, 6);
        edgelabels(0, 0, 6, 7, dir(-90), dir(-54));
        edgelabels(0, 5, 7, 10);
        edgelabels(5, 4, 8, 11);
        edgelabels(4, 0, 9, 15);
        edgelabels(0, 0, 10, 9, dir(-18), dir(18));
        edgelabels(0, 0, 11, 8, dir(54), dir(90));
        edgelabels(0, 1, 12, 16);
        edgelabels(1, 5, 13, 4);
        edgelabels(5, 2, 14, 5);
        edgelabels(2, 2, 15, 2, dir(105), dir(165));
        edgelabels(2, 1, 16, 3);

        rectangle((1.6, 1.6), (-1.2, -1.2));
    \end{asy}
    \caption{The bijection in \cref{lemma:bijectionstep1}.}
    \label{figure:bijectionstep1}
\end{figure}

\begin{proof}
    Consider an $(s_1, \dots, s_r)$-colored $n$-cycle with underlying cycle $\sigma$. As shown in step (1a), label the ``outside'' of each edge with the number it originates from, then label the ``inside'' of each edge in numerical order starting from the edge with outside label ``1''. As \cref{figure:bijectionstep1} suggests, the outside numbers will be referred to as ``circled numbers'', and the inside numbers will be referred to as ``uncircled numbers''. Treat the resulting object as a cycle digraph on $n$ colored vertices with doubly labeled edges.

    To construct the $(s_1, \dots, s_r)$-degree digraph from this colored cycle digraph with doubly labeled edges, identify vertices of the same color, as shown in step (1b). This gives a digraph with $r$ vertices such that vertex $i$ has degree $s_i$ for each $1 \leq i \leq r$. The uncircled numbers form an Eulerian tour by construction, and the circled numbers also form an Eulerian tour because the definition of an $r$-colored $n$-cycle forces $\sigma(\tau(i))$ to have the same color as $i$ for each $1 \leq i \leq n$.
    
    Since this process is reversible, this proves the bijection.
\end{proof}

\newpage

To state the second subpart of \cref{theorem:bijection}'s proof, define, given a directed graph such that the indegree and outdegree of each vertex are equal,
\begin{itemize}
    \item a \emph{wiring} of a vertex $v$ to be a pairing of the ingoing and outgoing edges incident to $v$,
    \item an \emph{exit ordering} of a vertex to be a total ordering of the outgoing edges incident to the vertex, and
    \item the \emph{last exits} of a set of vertices $V$, given an exit ordering for each vertex, to be the set of directed edges that appear last in the exit ordering of some vertex in $V$.
\end{itemize}
A wiring and an exit ordering are depicted in \cref{figure:wiring}. 

\begin{figure}[h]
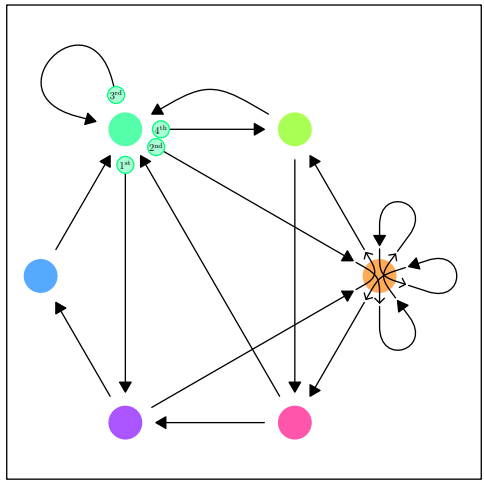

    \centering
    \begin{asy}[width=0.5\textwidth]
        drawvertices();

        drawedge(1, 2, dir(150), dir(30));
        drawedge(2, 4);
        drawedge(4, 3);
        drawedge(3, 2);
        drawedge(2, 0);
        drawedge(0, 0, dir(-90), dir(-54));
        drawedge(0, 5);
        drawedge(5, 4);
        drawedge(4, 0);
        drawedge(0, 0, dir(-18), dir(18));
        drawedge(0, 0, dir(54), dir(90));
        drawedge(0, 1);
        drawedge(1, 5);
        drawedge(5, 2);
        drawedge(2, 2, dir(105), dir(165));
        drawedge(2, 1);

        connector(0, dir(210), dir(120));
        connector(0, dir(150), dir(270));
        connector(0, dir(-54), dir(54));
        connector(0, dir(90), dir(-18));
        connector(0, dir(18), dir(240));

        exitcounter(2, 1, dir(-90));
        exitcounter(2, 2, dir(-30));
        exitcounter(2, 3, dir(105));
        exitcounter(2, 4, dir(0));

        rectangle((1.6, 1.6), (-1.2, -1.2));
    \end{asy}
    \caption{A wiring of the first (orange) vertex and an exit ordering of the third (sea green) vertex in the graph from \cref{figure:digraph}. The last exit of the third vertex is the edge pointing towards the second (lime green) vertex.}
    \label{figure:wiring}
\end{figure}

The second subpart of \cref{theorem:bijection}'s proof, depicted in \cref{figure:bijectionstep2}, is inspired by the BEST theorem and the ``last exit tree'':

\begin{lemma}\label{lemma:bijectionstep2}
    Let $s_1, \dots, s_r$ be a sequence of positive integers summing to $n$.
    There is a bijection between ordered pairs of Eulerian tours with the same starting edge on degree-$(s_1, \dots, s_r)$ digraphs and structures on degree-$(s_1, \dots, s_r)$ digraphs consisting of
    \begin{itemize}
        \item a choice of vertex $v \in \{1, 2, \dots, r\}$,
        \item a choice of wiring for each vertex, and
        \item a choice of exit orderings for each vertex
    \end{itemize}
    such that the wirings induce a cycle on the edges of the digraph, and the last exits of $\{1, 2, \dots, r\} \setminus \{v\}$ form a spanning tree directed towards $v$.
\end{lemma}

\begin{figure}[!p]
    \begin{minipage}{0.45\textwidth}
        \centering
        \begin{asy}[width=\textwidth]
            drawvertices();
    
            edgelabels(1, 2, 1, 1, dir(150), dir(30));
            edgelabels(2, 4, 2, 14);
            edgelabels(4, 3, 3, 12);
            edgelabels(3, 2, 4, 13);
            edgelabels(2, 0, 5, 6);
            edgelabels(0, 0, 6, 7, dir(-90), dir(-54));
            edgelabels(0, 5, 7, 10);
            edgelabels(5, 4, 8, 11);
            edgelabels(4, 0, 9, 15);
            edgelabels(0, 0, 10, 9, dir(-18), dir(18));
            edgelabels(0, 0, 11, 8, dir(54), dir(90));
            edgelabels(0, 1, 12, 16);
            edgelabels(1, 5, 13, 4);
            edgelabels(5, 2, 14, 5);
            edgelabels(2, 2, 15, 2, dir(105), dir(165));
            edgelabels(2, 1, 16, 3);

            rectangle((1.6, 1.6), (-1.2, -1.2));

            currentpicture = shift((0, 3.4)) * currentpicture;

            drawvertices();

            drawedge(1, 2, dir(150), dir(30));
            drawedge(2, 4);
            drawedge(4, 3);
            drawedge(3, 2);
            drawedge(2, 0);
            drawedge(0, 0, dir(-90), dir(-54));
            drawedge(0, 5);
            drawedge(5, 4);
            drawedge(4, 0);
            drawedge(0, 0, dir(-18), dir(18));
            drawedge(0, 0, dir(54), dir(90));
            drawedge(0, 1);
            drawedge(1, 5);
            drawedge(5, 2);
            drawedge(2, 2, dir(105), dir(165));
            drawedge(2, 1);

            connector(2, dir(30), dir(105));
            connector(2, dir(165), dir(0));
            connector(1, dir(180), dir(270));
            connector(5, dir(90), dir(120));
            connector(2, dir(-60), dir(-30));
            connector(0, dir(150), dir(270));
            connector(0, dir(-54), dir(54));
            connector(0, dir(90), dir(-18));
            connector(0, dir(18), dir(240));
            connector(5, dir(60), dir(180));
            connector(4, dir(0), dir(120));
            connector(3, dir(-60), dir(60));
            connector(2, dir(-120), dir(-90));
            connector(4, dir(90), dir(30));
            connector(0, dir(210), dir(120));
            connector(1, dir(-60), dir(150));

            currentpicture = shift((0, 2.5)) * currentpicture;

            label("$+$", (0.2, 1.3));

            drawvertices();
            label(scale(2/3)*"$v$", vertex(1));

            drawedge(1, 2, dir(150), dir(30));
            drawedge(2, 4);
            drawedge(4, 3);
            drawedge(3, 2, isthick=true);
            drawedge(2, 0);
            drawedge(0, 0, dir(-90), dir(-54));
            drawedge(0, 5);
            drawedge(5, 4);
            drawedge(4, 0, isthick=true);
            drawedge(0, 0, dir(-18), dir(18));
            drawedge(0, 0, dir(54), dir(90));
            drawedge(0, 1, isthick=true);
            drawedge(1, 5);
            drawedge(5, 2, isthick=true);
            drawedge(2, 2, dir(105), dir(165));
            drawedge(2, 1, isthick=true);

            exitcounter(0, 1, dir(270));
            exitcounter(0, 2, dir(240));
            exitcounter(0, 3, dir(-18));
            exitcounter(0, 4, dir(54));
            exitcounter(0, 5, dir(120));

            exitcounter(1, 1, dir(150));
            exitcounter(1, 2, dir(-90));

            exitcounter(2, 1, dir(-90));
            exitcounter(2, 2, dir(-30));
            exitcounter(2, 3, dir(105));
            exitcounter(2, 4, dir(0));

            exitcounter(3, 1, dir(60));

            exitcounter(4, 1, dir(120));
            exitcounter(4, 2, dir(30));

            exitcounter(5, 1, dir(180));
            exitcounter(5, 2, dir(120));

            rectangle((1.6, 2.5+1.6), (-1.2, -1.2));
        \end{asy}
    \end{minipage} \hfill
    \begin{minipage}{0.05\textwidth}
        \[\overset{\text{(2a)}}{\Longleftrightarrow}\]
        \vspace{12\baselineskip}
        \[\overset{\text{(2b)}}{\Longleftrightarrow}\]
        \vspace{12\baselineskip}
        \[\overset{\text{(2c)}}{\Longleftrightarrow}\]
    \end{minipage} \hfill
    \begin{minipage}{0.45\textwidth}
        \centering
        \begin{asy}[width=\textwidth]
            drawvertices();
            label(scale(2/3)*"$v$", vertex(1));

            edgelabel2(1, 2, 1, dir(150), dir(30));
            edgelabel2(2, 4, 14);
            edgelabel2(4, 3, 12);
            edgelabel2(3, 2, 13);
            edgelabel2(2, 0, 6);
            edgelabel2(0, 0, 7, dir(-90), dir(-54));
            edgelabel2(0, 5, 10);
            edgelabel2(5, 4, 11);
            edgelabel2(4, 0, 15);
            edgelabel2(0, 0, 9, dir(-18), dir(18));
            edgelabel2(0, 0, 8, dir(54), dir(90));
            edgelabel2(0, 1, 16);
            edgelabel2(1, 5, 4);
            edgelabel2(5, 2, 5);
            edgelabel2(2, 2, 2, dir(105), dir(165));
            edgelabel2(2, 1, 3);

            currentpicture = shift((0, 2.5)) * currentpicture;

            label("$+$", (0.2, 1.3));
            
            drawvertices();
            label(scale(2/3)*"$v$", vertex(1));

            edgelabel1(1, 2, 1, dir(150), dir(30));
            edgelabel1(2, 4, 2);
            edgelabel1(4, 3, 3);
            edgelabel1(3, 2, 4);
            edgelabel1(2, 0, 5);
            edgelabel1(0, 0, 6, dir(-90), dir(-54));
            edgelabel1(0, 5, 7);
            edgelabel1(5, 4, 8);
            edgelabel1(4, 0, 9);
            edgelabel1(0, 0, 10, dir(-18), dir(18));
            edgelabel1(0, 0, 11, dir(54), dir(90));
            edgelabel1(0, 1, 12);
            edgelabel1(1, 5, 13);
            edgelabel1(5, 2, 14);
            edgelabel1(2, 2, 15, dir(105), dir(165));
            edgelabel1(2, 1, 16);

            rectangle((1.6, 2.5+1.6), (-1.2, -1.2));

            currentpicture = shift((0, 3.4)) * currentpicture;

            rectangle((1.6, 1.6), (-1.2, -1.2));

            currentpicture = scale(0.95) * currentpicture;
            
            drawvertices();

            drawedge(1, 2, dir(150), dir(30));
            drawedge(2, 4);
            drawedge(4, 3);
            drawedge(3, 2, isthick=true);
            drawedge(2, 0);
            drawedge(0, 0, dir(-90), dir(-54));
            drawedge(0, 5);
            drawedge(5, 4);
            drawedge(4, 0, isthick=true);
            drawedge(0, 0, dir(-18), dir(18));
            drawedge(0, 0, dir(54), dir(90));
            drawedge(0, 1, isthick=true);
            drawedge(1, 5);
            drawedge(5, 2, isthick=true);
            drawedge(2, 2, dir(105), dir(165));
            drawedge(2, 1, isthick=true);
    
            connector(2, dir(30), dir(105));
            connector(2, dir(165), dir(0));
            connector(1, dir(180), dir(270));
            connector(5, dir(90), dir(120));
            connector(2, dir(-60), dir(-30));
            connector(0, dir(150), dir(270));
            connector(0, dir(-54), dir(54));
            connector(0, dir(90), dir(-18));
            connector(0, dir(18), dir(240));
            connector(5, dir(60), dir(180));
            connector(4, dir(0), dir(120));
            connector(3, dir(-60), dir(60));
            connector(2, dir(-120), dir(-90));
            connector(4, dir(90), dir(30));
            connector(0, dir(210), dir(120));
            connector(1, dir(-60), dir(150));
    
            exitcounter(0, 1, dir(270));
            exitcounter(0, 2, dir(240));
            exitcounter(0, 3, dir(-18));
            exitcounter(0, 4, dir(54));
            exitcounter(0, 5, dir(120));
    
            exitcounter(1, 1, dir(150));
            exitcounter(1, 2, dir(-90));
    
            exitcounter(2, 1, dir(-90));
            exitcounter(2, 2, dir(-30));
            exitcounter(2, 3, dir(105));
            exitcounter(2, 4, dir(0));
    
            exitcounter(3, 1, dir(60));
    
            exitcounter(4, 1, dir(120));
            exitcounter(4, 2, dir(30));
    
            exitcounter(5, 1, dir(180));
            exitcounter(5, 2, dir(120));

            fill(circle(vertex(1), 0.05), white+opacity(2/3));
            label(scale(2/3)*"$v$", vertex(1));
        \end{asy}
    \end{minipage}
    \caption{The bijection in \cref{lemma:bijectionstep2}.}
    \label{figure:bijectionstep2}
\end{figure} \clearpage

\begin{proof}
    Consider a pair of Eulerian tours on a $(s_1, \dots, s_r)$-degree digraph. As shown in step (2a), treat the tours as separate, and call the origin of the common starting edge vertex $v$. As shown in step (2b), wire each vertex according to the first tour, and order the exits around each vertex according the traversal order of the second tour. Combining these gives a structure consisting of a choice of $v$, wirings, and exit orders, as shown in step (2c).

    To prove that this gives the desired bijection, it suffices to check that
    \begin{itemize}
        \item[(i)] the last exits of $\{1, 2, \dots, r\} \setminus \{v\}$ always form a spanning tree directed towards $v$, and
        \item[(ii)] every choice of exit orderings such that the last exits of $\{1, 2, \dots, r\} \setminus \{v\}$ form a spanning tree directed towards $v$ gives rise to an Eulerian tour.
    \end{itemize}
    Indeed, (i) is true because
    \begin{itemize}
        \item the last exits of $\{1, 2, \dots, r\}$ must form a subgraph with $r-1$ edges,
        \item this subgraph cannot contain any cycles because the exit orderings were induced from an Eulerian tour, and
        \item every edge is directed towards $v$ by construction.
    \end{itemize}
    (ii) is true because the Eulerian tour can be obtained by starting at $v$ and repeatedly traversing the first untraversed edge in each exit ordering.
\end{proof}

Finally, the third subpart of \cref{theorem:bijection}'s proof, depicted in \cref{figure:bijectionstep3}, is motivated by the Pr\"ufer sequence:
\begin{lemma}\label{lemma:bijectionstep3}
    Let $s_1, \dots, s_r$ be a sequence of positive integers summing to $n$, and let $S$ be the $n$-element set
    \[\{(1, 1), (1, 2), \dots, (1, s_1), (2, 1), (2, 2), \dots, (2, s_2), \dots, (r, 1), (r, 2), \dots, (r, s_r)\}.\]
    There is a bijection between structures on degree-$(s_1, \dots, s_r)$ digraphs consisting of
    \begin{itemize}
        \item a choice of vertex $v \in \{1, 2, \dots, r\}$,
        \item a choice of wiring for each vertex, and 
        \item a choice of exit orderings for each vertex
    \end{itemize}
    such that the wirings induce a cycle on the digraph's edges and the last exits of $\{1, 2, \dots, v\} \setminus \{v\}$ form a spanning tree directed towards $v$, and ways to arrange the elements of $S$ to form an $(r-1)$-term sequence and an $(n-r+1)$-element cycle.
\end{lemma}

\begin{proof}
    Consider such a structure on a degree-$(s_1, \dots, s_r)$ digraph. As shown in step (3a), follow the wiring and skip the label in front of each tree edge to form an $(n-r+1)$-cycle of elements in $S$. Next, delete each non-tree edge. Then, as shown in step (3b), ignore the labels not associated with any tree edges. Finally, construct an $(r-1)$-element sequence of $S$ by repeatedly deleting the smallest-index leaf and adding the corresponding element of $S$ to the sequence, as shown in step (3c).

    To show that this process is indeed a bijection, it suffices to describe the inverse of the digraph-to-sequence process. To do this, start with an $(r-1)$-element sequence of elements in $S$, and repeat the following algorithm $r-1$ times:
    \begin{itemize}
        \item Find the smallest-index color, $i$, that does not appear in the sequence.
        \item Draw an arrow from vertex $i$ to the color of the first element of the sequence, labeled by the corresponding ordinal.
        \item Delete the first element of the sequence. \qedhere
    \end{itemize}
\end{proof}

\begin{figure}[p]
    \centering
    \begin{minipage}{0.45\textwidth}
        \centering
        \begin{asy}[width=\textwidth]
            drawvertices();

            drawedge(1, 2, dir(150), dir(30));
            drawedge(2, 4);
            drawedge(4, 3);
            drawedge(3, 2, isthick=true);
            drawedge(2, 0);
            drawedge(0, 0, dir(-90), dir(-54));
            drawedge(0, 5);
            drawedge(5, 4);
            drawedge(4, 0, isthick=true);
            drawedge(0, 0, dir(-18), dir(18));
            drawedge(0, 0, dir(54), dir(90));
            drawedge(0, 1, isthick=true);
            drawedge(1, 5);
            drawedge(5, 2, isthick=true);
            drawedge(2, 2, dir(105), dir(165));
            drawedge(2, 1, isthick=true);

            connector(2, dir(30), dir(105));
            connector(2, dir(165), dir(0));
            connector(1, dir(180), dir(270));
            connector(5, dir(90), dir(120));
            connector(2, dir(-60), dir(-30));
            connector(0, dir(150), dir(270));
            connector(0, dir(-54), dir(54));
            connector(0, dir(90), dir(-18));
            connector(0, dir(18), dir(240));
            connector(5, dir(60), dir(180));
            connector(4, dir(0), dir(120));
            connector(3, dir(-60), dir(60));
            connector(2, dir(-120), dir(-90));
            connector(4, dir(90), dir(30));
            connector(0, dir(210), dir(120));
            connector(1, dir(-60), dir(150));

            exitcounter(0, 1, dir(270));
            exitcounter(0, 2, dir(240));
            exitcounter(0, 3, dir(-18));
            exitcounter(0, 4, dir(54));
            exitcounter(0, 5, dir(120));

            exitcounter(1, 1, dir(150));
            exitcounter(1, 2, dir(-90));

            exitcounter(2, 1, dir(-90));
            exitcounter(2, 2, dir(-30));
            exitcounter(2, 3, dir(105));
            exitcounter(2, 4, dir(0));

            exitcounter(3, 1, dir(60));

            exitcounter(4, 1, dir(120));
            exitcounter(4, 2, dir(30));

            exitcounter(5, 1, dir(180));
            exitcounter(5, 2, dir(120));

            rectangle((-1.2, -1.2), (1.6, 1.6));

            fill(circle(vertex(1), 0.05), white+opacity(2/3));
            label(scale(2/3)*"$v$", vertex(1));

            currentpicture = shift((0, 3.2)) * currentpicture;
            currentpicture = shift((-0.2, 0)) * currentpicture;
            
            drawvertices();
            label(scale(2/3)*"$v$", vertex(1));
            
            drawedge(3, 2, isthick=true);
            drawedge(4, 0, isthick=true);
            drawedge(0, 1, isthick=true);
            drawedge(5, 2, isthick=true);
            drawedge(2, 1, isthick=true);

            exitcounter(0, 5, 1.5*dir(210));
            exitcounter(1, 1, 1.5*dir(180));
            exitcounter(1, 2, 1.5*dir(-60));
            exitcounter(2, 1, 1.5*dir(-120));
            exitcounter(2, 2, 1.5*dir(-60));

            currentpicture = scale(1.1/1.4) * currentpicture;
            drawcycle(new int[] {3, 4, 2, 1, 4, 3, 2, 1, 1, 1, 2}, new int[] {2, 2, 5, 0, 0, 0, 0, 5, 4, 3, 4}, c =(0, -2), s=2/3, card=true, labelsize=2/3, colorboundary=true);
            currentpicture = scale(1.4/1.1) * currentpicture;

            label("$+$", (0, -1.25));

            rectangle((-1.4, -3.7), (1.4, 1.4));
        \end{asy}
    \end{minipage} \hfill
    \begin{minipage}{0.05\textwidth}
        \[\overset{\text{(3a)}}{\Longleftrightarrow}\]
        \vspace{11\baselineskip}
        \[\overset{\text{(3b)}}{\Longleftrightarrow}\]
        \vspace{11\baselineskip}
        \[\overset{\text{(3c)}}{\Longleftrightarrow}\]
    \end{minipage} \hfill
    \begin{minipage}{0.45\textwidth}
        \centering
        \begin{asy}[width=\textwidth]
            drawvertices();
            
            drawedge(3, 2, isthick=true);
            drawedge(4, 0, isthick=true);
            drawedge(0, 1, isthick=true);
            drawedge(5, 2, isthick=true);
            drawedge(2, 1, isthick=true);

            connector(2, dir(30), dir(105));
            connector(2, dir(165), dir(0));
            connector(1, dir(180), dir(270));
            connector(5, dir(90), dir(120));
            connector(2, dir(-60), dir(-30));
            connector(0, dir(150), dir(270));
            connector(0, dir(-54), dir(54));
            connector(0, dir(90), dir(-18));
            connector(0, dir(18), dir(240));
            connector(5, dir(60), dir(180));
            connector(4, dir(0), dir(120));
            connector(3, dir(-60), dir(60));
            connector(2, dir(-120), dir(-90));
            connector(4, dir(90), dir(30));
            connector(0, dir(210), dir(120));
            connector(1, dir(-60), dir(150));

            exitcounter(0, 1, dir(270));
            exitcounter(0, 2, dir(240));
            exitcounter(0, 3, dir(-18));
            exitcounter(0, 4, dir(54));
            exitcounter(0, 5, dir(120));

            exitcounter(1, 1, dir(150));
            exitcounter(1, 2, dir(-90));

            exitcounter(2, 1, dir(-90));
            exitcounter(2, 2, dir(-30));
            exitcounter(2, 3, dir(105));
            exitcounter(2, 4, dir(0));

            exitcounter(3, 1, dir(60));

            exitcounter(4, 1, dir(120));
            exitcounter(4, 2, dir(30));

            exitcounter(5, 1, dir(180));
            exitcounter(5, 2, dir(120));

            fill(circle(vertex(1), 0.05), white+opacity(2/3));
            label(scale(2/3)*"$v$", vertex(1));

            currentpicture = scale(1.1/1.4) * currentpicture;
            drawcycle(new int[] {3, 4, 2, 1, 4, 3, 2, 1, 1, 1, 2}, new int[] {2, 2, 5, 0, 0, 0, 0, 5, 4, 3, 4}, c =(0, -2), s=2/3, card=true, labelsize=2/3, colorboundary=true);
            currentpicture = scale(1.4/1.1) * currentpicture;

            label("$+$", (0, -1.25));

            rectangle((-1.4, -3.7), (1.4, 1.4));
            
            currentpicture = shift((0, 4.55)) * currentpicture;
            currentpicture = scale(1.1/1.4) * currentpicture;

            numbercircle(2, 1, (-0.6, 0));
            numbercircle(0, 5, (-0.3, 0));
            numbercircle(1, 2, (0, 0));
            numbercircle(2, 2, (0.3, 0));
            numbercircle(1, 1, (0.6, 0)); 

            drawcycle(new int[] {3, 4, 2, 1, 4, 3, 2, 1, 1, 1, 2}, new int[] {2, 2, 5, 0, 0, 0, 0, 5, 4, 3, 4}, c =(0, -7/6), s=2/3, card=true, labelsize=2/3, colorboundary=true);

            label("$+$", (0, -0.25));

            rectangle((-1.1, -2), (1.1, 0.2));
        \end{asy}
    \end{minipage}
    \caption{The bijection in \cref{lemma:bijectionstep3}.}
    \label{figure:bijectionstep3}
\end{figure} \clearpage

Combining these lemmas gives:

\begin{proof}[Proof of \cref{theorem:bijection}]
    Chaining \cref{lemma:bijectionstep1}, \cref{lemma:bijectionstep2}, and \cref{lemma:bijectionstep3} gives the desired result.
\end{proof}

Now $(s_1, \dots, s_r)$-colored $n$-cycles and $r$-colored $n$-cycles can be counted via the next two corollaries:
\begin{corollary}\label{corollary:seqcoloredcycle}
    Let $s_1, \dots, s_r$ be a sequence of positive integers summing to $n$. Then
    \[\#\{\text{$(s_1, \dots, s_r)$-colored $n$-cycles}\}=\frac{n!}{n-r+1}.\]
\end{corollary}

\begin{proof}
    By \cref{theorem:bijection}, it suffices to count the number of ways to arrange an $n$-element set into an $(r-1)$-term sequence and an $(n-r+1)$-element cycle. There is a 1-to-$(n-r+1)$ map from such ways to permutations of $n$ elements, given by unwrapping the cycle at one of $n-r+1$ possible locations and appending it to the $(r-1)$-term sequence. This gives the desired result.
\end{proof}

\begin{corollary}\label{corollary:coloredcycle}
    Let $r$ and $n$ be positive integers. Then
    \[\#\{\text{$r$-colored $n$-cycles}\} = \binom{n-1}{r-1} \frac{n!}{n-r+1}.\]
\end{corollary}

\begin{proof}
    The number of ways to choose $r$ positive integers $s_1, \dots, s_r$ summing to $n$ is $\binom{n-1}{r-1}$, so using \cref{corollary:seqcoloredcycle} gives the result.
\end{proof}

Finally, the bijection in \cref{theorem:bijection} can be extended to count $t$-colored $k$-subsets of $r$-colored $n$-cycles.
To do this, define an \emph{$r$-colored $n$-strip} to be a coloring of $\{1, 2, \dots, n\}$ with $r$ colors in nondecreasing order, and define a \emph{$t$-colored $k$-subset of an $r$-colored $n$-strip} to be an $r$-colored $n$-strip and a $k$-element subset of $\{1, 2, \dots, n\}$ colored with $t$ distinct colors. An example of a 6-colored 16-strip is given in \cref{figure:coloredstrip}, and an example of a 3-color 5-subset of a 6-colored 16-strip is given in \cref{figure:coloredstripsubset}.

\begin{figure}[ht]
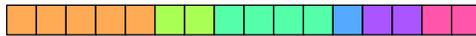

    \centering
    \begin{asy}[width = 0.5\textwidth]
        int boxcolors[] = {0, 0, 0, 0, 0, 1, 1, 2, 2, 2, 2, 3, 4, 4, 5, 5};
        for (int i = 0; i < 16; ++i) {
            path p = (i, 0) -- (i, 1) -- (i+1, 1) -- (i+1, 0) -- cycle;
            filldraw(p, rainbow(boxcolors[i]/6, 1/3));
        }
    \end{asy}
    \caption{A 6-colored 16-strip.}
    \label{figure:coloredstrip}
\end{figure}

\begin{figure}[ht]
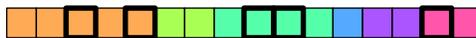

    \centering
    \begin{asy}[width = 0.5\textwidth]
        int boxcolors[] = {0, 0, 0, 0, 0, 1, 1, 2, 2, 2, 2, 3, 4, 4, 5, 5};
        int boxspecial[] = {2, 4, 8, 9, 14};
        for (int i = 0; i < 16; ++i) {
            path p = (i, 0) -- (i, 1) -- (i+1, 1) -- (i+1, 0) -- cycle;
            filldraw(p, rainbow(boxcolors[i]/6, 1/3));
        }
        for (int j = 0; j < boxspecial.length; ++j) {
            int i = boxspecial[j];
            path p = (i, 0) -- (i, 1) -- (i+1, 1) -- (i+1, 0) -- cycle;
            draw(p, linewidth(2));
        }
    \end{asy}
    \caption{A 3-color 5-subset of a 6-colored 16-strip.}
    \label{figure:coloredstripsubset}
\end{figure}

The bijection extension, depicted in \cref{figure:extendedbijection} and \cref{figure:coloredstripcounting}, is given by the following two lemmas:

\begin{lemma}\label{lemma:extendedbijection}
    Let $r$, $n$, $t$, and $k$ be positive integers. There is a bijection between $t$-colored $k$-subsets of an $r$-colored $n$-cycle and structures consisting of
    \begin{itemize}
        \item a $t$-colored $k$-subset of an $r$-colored $n$-strip, and
        \item an $(s_1, \dots, s_r)$-colored $n$-cycle, where $s_i$ is the number of elements in the $r$-colored $n$-strip with color $i$ for each $1 \leq i \leq r$.
    \end{itemize}
\end{lemma}

\newpage

\begin{figure}[h]
    \centering
    \begin{minipage}{0.45\textwidth}
        \centering
        \begin{asy}[width=\textwidth]
            int p[] = {1, 14, 12, 13, 6, 7, 10, 11, 15, 9, 8, 16, 4, 5, 2, 3};
            int colors[] = {1, 2, 4, 3, 2, 0, 0, 5, 4, 0, 0, 0, 1, 5, 2, 2};
            drawcycle(p, colors, special = new int[] {2, 6, 9, 11, 16}, labelsize=4/5);

            rectangle((1.2, 1.4), (-1.2, -1.4));
        \end{asy}
    \end{minipage}
    \begin{minipage}{0.05\textwidth}
        \[\overset{\text{(4a)}}{\Longrightarrow}\]
    \end{minipage}
    \begin{minipage}{0.45\textwidth}
        \centering
        \begin{asy}[width=\textwidth]
            int p[] = {1, 14, 12, 13, 6, 7, 10, 11, 15, 9, 8, 16, 4, 5, 2, 3};
            int colors[] = {1, 2, 4, 3, 2, 0, 0, 5, 4, 0, 0, 0, 1, 5, 2, 2};
            drawcycle(p, colors, labelsize=4/5);
            
            int boxcolors[] = {0, 0, 0, 0, 0, 1, 1, 2, 2, 2, 2, 3, 4, 4, 5, 5};
            int boxspecial[] = {2, 4, 8, 9, 14};
            for (int i = 0; i < 16; ++i) {
                pair P = (-1 + i/8, -1.4);
                real l = 1/8;
                path p = P -- P+(l,0) -- P+(l,l) -- P+(0,l) -- cycle;
                filldraw(p, rainbow(boxcolors[i]/6, 1/3));
            }
            for (int j = 0; j < boxspecial.length; ++j) {
                int i = boxspecial[j];
                pair P = (-1 + i/8, -1.4);
                real l = 1/8;
                path p = P -- P+(l,0) -- P+(l,l) -- P+(0,l) -- cycle; 
                draw(p, linewidth(2));
            }
            rectangle((1.2, 1.2), (-1.2, -1.6));
        \end{asy}
    \end{minipage}
    \caption{The bijection in \cref{lemma:extendedbijection}.}
    \label{figure:extendedbijection}
\end{figure}

\begin{proof}
    Consider a $t$-colored $k$-subset of an $r$-colored $n$-cycle, as shown in \cref{figure:extendedbijection}. 
    Cutting it before the ``1'', removing the numbers, and sorting the colors while preserving the relative order of the elements in the subset gives a $t$-colored $k$-subset of a $r$-colored $n$-strip. Forgetting about the subset gives the $(s_1, \dots, s_r)$-colored $n$-cycle, where $s_i$ is the number of elements in the $r$-colored $n$-cycle with color $i$ for each $1 \leq i \leq r$. This process is reversible, so it is a bijection.
\end{proof}

\begin{lemma}\label{lemma:coloredstrip}
    There is a bijection between $t$-colored $k$-subsets of $r$-colored $n$-strips and structures consisting of an $(r+k)$-colored $(n+t)$-strip, a $t$-element subset of $\{1, 2, \dots, r\}$, and a $t$-colored $k$-strip.
\end{lemma}

\begin{figure}[ht]
    \centering
    \begin{asy}[width = 0.5\textwidth]
        int boxcolors[] = {0, 0, 0, 0, 0, 1, 1, 2, 2, 2, 2, 3, 4, 4, 5, 5};
        int boxspecial[] = {2, 4, 8, 9, 14};
        for (int i = 0; i < 16; ++i) {
            path p = (i, 0) -- (i, 1) -- (i+1, 1) -- (i+1, 0) -- cycle;
            filldraw(p, rainbow(boxcolors[i]/6, 1/3));
        }
        for (int j = 0; j < boxspecial.length; ++j) {
            int i = boxspecial[j];
            path p = (i, 0) -- (i, 1) -- (i+1, 1) -- (i+1, 0) -- cycle;
            draw(p, linewidth(2));
        }
    \end{asy}
    \[\Downarrow\text{\scriptsize{(4b)}}\]
    \begin{asy}[width = 0.59375\textwidth]
        int boxcolors[] = {0, 0, 0, 0, 0, 0, 1, 1, 2, 2, 2, 2, 2, 3, 4, 4, 5, 5, 5};
        int labels[] = {0, 0, 0, 1, 1, 2, 0, 0, 0, 0, 1, 2, 2, 0, 0, 0, 0, 1, 1};
        for (int i = 0; i < 19; ++i) {
            path p = (i, 0) -- (i, 1) -- (i+1, 1) -- (i+1, 0) -- cycle;
            filldraw(p, rainbow(boxcolors[i]/6, 1/3));
            label(scale(3/4)*string(labels[i]), (i+1/2, 1/2));
        }
    \end{asy}
    \[\Downarrow\text{\scriptsize{(4c)}}\]
    \begin{asy}[width = 0.96875\textwidth]
        int labels[] = {1, 1, 1, 2, 2, 3, 4, 4, 5, 5, 6, 7, 7, 8, 9, 9, 10, 11, 11};
        for (int i = 0; i < labels.length; ++i) {
            path p = (i, 0) -- (i, 1) -- (i+1, 1) -- (i+1, 0) -- cycle;
            draw(p);
            label(scale(3/4)*string(labels[i]), (i+1/2, 1/2));
        }
        currentpicture = shift((-21, 0)) * currentpicture;
        label("$+$", (-1, 1/2));
        for (int i = 0; i < 6; ++i) {
            path p = (i, 0) -- (i, 1) -- (i+1, 1) -- (i+1, 0) -- cycle;
            draw(p);
            filldraw(p, rainbow(i/6, 1/3));
        }
        int[] boxspecial = {0, 2, 5};
        for (int j = 0; j < boxspecial.length; ++j) {
            int i = boxspecial[j];
            path p = (i, 0) -- (i, 1) -- (i+1, 1) -- (i+1, 0) -- cycle;
            draw(p, linewidth(2));
        }
        currentpicture = shift((-8, 0)) * currentpicture;
        label("$+$", (-1, 1/2));
        int boxcolors[] = {1, 1, 2, 2, 3};
        for (int i = 0; i < boxcolors.length; ++i) {
            path p = (i, 0) -- (i, 1) -- (i+1, 1) -- (i+1, 0) -- cycle;
            draw(p);
            label(scale(3/4)*string(boxcolors[i]), (i+1/2,1/2));
        }
    \end{asy}
    \caption{The bijection in \cref{lemma:coloredstrip}.}
    \label{figure:coloredstripcounting}
\end{figure}

\begin{proof}
    Given a $t$-colored $k$-subset of an $r$-colored $n$-strip, label each box with the number of selected boxes of the same color before it, including itself. Additionally, for each color with a selected box, add an additional box at the beginning. This is shown in step (4b) of \cref{figure:coloredstripcounting}.

    Now:
    \begin{itemize}
        \item let the $(r+k)$-colored $(n+t)$-strip be obtained by forgetting about the colors in the resulting structure and renumbering in ascending order,
        \item let the $t$-element subset of $\{1, 2, \dots, r\}$ be the set of colors in the $t$-colored $k$-subset, and
        \item let the block lengths of the $t$-colored $k$-strip represent the multiplicity of these $t$ colors. 
    \end{itemize}
    This is shown in step (4c) of \cref{figure:coloredstripcounting}, and completes the bijection description.

    This process gives a valid bijection because the process has an inverse: use the $t$-element subset of $\{1, 2, \dots, r\}$ and the $t$-colored $k$-strip to construct the ``colored-numbered strip'' shown in the middle of \cref{figure:coloredstripcounting}, then delete a box labeled ``0'' for each color that contains a nonzero number.
\end{proof}

\begin{corollary}\label{corollary:coloredcyclesubset}
    The total number of $t$-colored $k$-subsets of $r$-colored $n$-cycles is
    \[\binom{n+t-1}{r+k-1} \binom{r}{t} \binom{k-1}{t-1} \frac{n!}{n-r+1}.\]
\end{corollary}

\begin{proof}
    Combining \cref{lemma:extendedbijection}, \cref{corollary:coloredcycle}, and \cref{lemma:coloredstrip} gives
    the result.
\end{proof}

\section{Coalescence probability}\label{section:coalescence}

This section will culminate in the proof of the coalescence probability formula, \cref{theorem:coalescence}. The proof is dense and technical; at key points in the proof, the general strategy will be highlighted to increase readability. \cref{lemma:coalescence} and \cref{corollary:coloredcyclesubset} shows that the probability that $1, 2, \dots, m$ are all contained in the same cycle of a product of two random $n$-cycles is
\[\frac{(-1)^n}{(n-1)! \binom{n}{k}}\sum_{t=1}^k \frac{1}{t} \sum_{r=1}^n (-1)^r \binom{n+t-1}{r+k-1} \binom{r}{t} \binom{k-1}{t-1} \frac{n!}{n-r+1},\]
so the proof of \cref{theorem:coalescence} will consist of showing that this expression is equal to
\[\frac{1}{k} + \frac{4 (-1)^n}{ \binom{2k}{k}} \sum_{\substack{1 \leq i \leq k-1 \\ i \not\equiv n \bmod 2}} \binom{2k-1}{k+i} \left(\frac{1}{n+i+1} - \frac{1}{n-i}\right).\]
The proof will require several combinatorial identities (\cref{lemma:identity1} to \ref{lemma:identity4}), the statements and proofs of which will be deferred to the end of this section.

\begin{proof}[Proof of \cref{theorem:coalescence}]
    The first idea is to write the starting expression in a form that absorbs the $\frac{1}{n-r+1}$ factor into a binomial. This is done is by first pushing as many terms outside the summation as possible, then breaking apart the $\smash{\binom{n+t-1}{r+k-1}}$ binomial so that the $n-r+1$ in the numerator cancels with the $\frac{1}{n-r+1}$ factor. Doing this gives
    \begin{align*}
        &\phantom{{}={}}\frac{(-1)^n}{(n-1)! \binom{n}{k}}\sum_{t=1}^k \frac{1}{t} \sum_{r=1}^n (-1)^r \binom{n+t-1}{r+k-1} \binom{r}{t} \binom{k-1}{t-1} \frac{n!}{n-r+1}\\
        &= (-1)^n \frac{n}{\binom{n}{k}}\sum_{t=1}^k \frac{1}{t}\binom{k-1}{t-1} \sum_{r=1}^n (-1)^r \binom{r}{t} \binom{n+t-1}{r+k-1} \frac{1}{n-r+1} \\
        &= (-1)^n \frac{n \binom{n+k}{k}}{\binom{2k}{k}\binom{n+k}{2k}}\sum_{t=1}^k \frac{k-t+1}{kt} \binom{k}{t-1} \sum_{r=1}^n (-1)^r \binom{r}{t}\frac{t\binom{n+k}{r+k-1}\binom{n-r}{k-t}\binom{n+t-1}{t}}{n(k-t+1)\binom{n+k}{k} \binom{k}{t-1}} \\
        &= \frac{(-1)^n}{k\binom{2k}{k}\binom{n+k}{2k}}\sum_{t=1}^k \binom{n+t-1}{t} \sum_{r=1}^n (-1)^r \binom{r}{t}\binom{n+k}{r+k-1}\binom{n-r}{k-t}.
    \end{align*}
    Now, the next goal is to write the expression in a form that makes it obvious that the expression is a rational function in $n$, up to a factor of $(-1)^n$. The idea here is to observe that if the innermost sum's bounds were extended to $1-k$ and $n+1$, then the resulting sum is zero. This observation is proved in \cref{lemma:identity1}. This allows one to replace the innermost summation with the negative of its ``complement'', as follows:
    \begin{align*}
        &\phantom{{}={}}\frac{(-1)^n}{k\binom{2k}{k}\binom{n+k}{2k}}\sum_{t=1}^k \binom{n+t-1}{t} \sum_{r=1}^n (-1)^r \binom{r}{t}\binom{n+k}{r+k-1}\binom{n-r}{k-t}\\
        &\overset{\mathclap{(\ref{lemma:identity1})}}{=}\frac{(-1)^n}{k\binom{2k}{k}\binom{n+k}{2k}}\sum_{t=1}^k \binom{n+t-1}{t} \sum_{\substack{r=n+1 \text{ or}\\ 1-k \leq r \leq 0}} -(-1)^r \binom{r}{t}\binom{n+k}{r+k-1}\binom{n-r}{k-t}.
    \end{align*}
    The innermost sum now consists of an ``upper'' part and a ``lower'' part, and the expression is now more clearly a rational function in $n$, up to a factor of $(-1)^n$. Splitting the sum up into these parts shows that the expression is equal to
    \[\frac{A(n, k) - (-1)^{n}B(n, k)}{k \binom{2k}{k}},\]
    where $A(n, k)$ is the summand at $r=n+1$,
    \begin{align*}
        A(n,k)&=\frac{(-1)^n}{\binom{n+k}{2k}} \sum_{t=1}^k \binom{n+t-1}{t} (-1)^n \binom{n+1}{t} \binom{n+k}{n+k} \binom{-1}{k-t} \\
        &=\frac{(-1)^k}{\binom{n+k}{2k}}\sum_{t=1}^k (-1)^t\binom{n+t-1}{t} \binom{n+1}{t},
    \end{align*}
    and $B(n, k)$ is the expression
    \begin{align*}
        B(n&, k) = \frac{1}{\binom{n+k}{2k}} \sum_{t=1}^k \binom{n+t-1}{t} \sum_{r'=0}^{k-1} (-1)^r \binom{-r'}{t} \binom{n+k}{k-1-r'}\binom{n+r'}{k-t} \\
        &= \frac{1}{\binom{n+k}{2k}}\sum_{t=1}^k (-1)^t \binom{n+t-1}{t} \sum_{r'=0}^{k-1} (-1)^{r'}\binom{r'+t-1}{t}\binom{n+k}{k-1-r'}\binom{n+r'}{k-t}.
    \end{align*}
    
    The next idea is to treat $k$ as fixed and write both $A$ and $B$ in a partial fraction decomposition form. The denominators of $A$ and $B$ are both degree-$2k$ polynomials in $n$ with roots at $k-1, k-2, \dots, -(k-1), -k$, so the coefficients of the partial fraction decomposition can be extracted by substituting $p$ into $(n-p) A(n, k)$ and $(n-p) B(n, k)$ for each root $p$.

    Let's apply this process to $A$ first. By inspection, $A$ can be written as a rational function with a degree-$2k$ numerator and a degree-$2k$ denominator. Thus $A$ is asymptotically constant with constant term
    \[\lim_{n \to \infty} A(n, k) = \lim_{n \to \infty} (-1)^k\frac{(-1)^k \binom{n+k-1}{k} \binom{n+1}{k}}{\binom{n+k}{2k}} = \binom{2k}{k}.\]
    The remaining terms of $A$ can be obtained by substituting $p$ into $(n-p)A(n, k)$ for each $p \in \{k-1, k-2, \dots, -k\}$. Doing this gives
    \begin{align*}
        A(n, k) &= \binom{2k}{k} + (-1)^k \sum_{p=-k}^{k-1} \frac{1}{n-p}\left(\left.\frac{n-p}{\binom{n+k}{2k}}\right|_{n=p} \right)\sum_{t=1}^k (-1)^t\binom{p+t-1}{t} \binom{p+1}{t} \\
        &= \binom{2k}{k} - 2k \sum_{p=-k}^{k-1} \frac{1}{n-p} (-1)^p \binom{2k-1}{k+p} \sum_{t=1}^{k} (-1)^t \binom{p+t-1}{t} \binom{p+1}{t}.
    \end{align*}
    Observe that if the bounds of the innermost sum ran from 0 to $\infty$, the sum would be zero unless $p \in \{-1, 0\}$, in which case the sum would equal one; this observation is \cref{lemma:identity2}. Since all terms with $t \geq k+1$ are zero, the sum can be replaced with the negation of its value at $t=0$, except at $p \in \{-1, 0\}$, for which it can be replaced with one plus this value. Since the summand's value at $t=0$ is 1, this gives
    \begin{align*}
        A(n, k) &= \binom{2k}{k} - 2k \sum_{p=-k}^{k-1} \frac{1}{n-p} (-1)^p \binom{2k-1}{k+p} (- \mathds{1}_{-k \leq p \leq -2\text{ or }1 \leq p \leq k-1}) \\
        &= \binom{2k}{k} + 2k \sum_{p=1}^{k-1} \frac{1}{n-p} (-1)^p \binom{2k-1}{k+p} + 2k \sum_{p'=2}^{k} \frac{1}{n+p'} (-1)^{p'} \binom{2k-1}{k-p'} \\
        &= \binom{2k}{k} + 2k \sum_{i=1}^{k-1} (-1)^i \binom{2k-1}{k+i} \left(\frac{1}{n-i} - \frac{1}{n+i+1}\right),
    \end{align*}
    which is the partial fraction decomposition of $A$.

    Repeating this process on $B$ is more involved. The same partial fraction decomposition trick gives
    \begin{align*}
        B(n, k)&= \sum_{p=-k}^{k-1} \frac{1}{n-p} \left(\left.\frac{n-p}{\binom{n+k}{2k}}\right|_{n=p}\right)B_i(k) \\
        &= (-1)^{k+1} 2k \sum_{p=-k}^{k-1} \frac{1}{n-p} (-1)^{p} \binom{2k-1}{k+p}B_i(k)
    \end{align*}
    where
    \[B_i(k) = \sum_{t=1}^k (-1)^t \binom{p+t-1}{t} \sum_{r'=0}^{k-1} (-1)^{r'}\binom{r'+t-1}{t}\binom{p+k}{k-1-r'}\binom{p+r'}{k-t}.\]
    To evaluate $B_i(k)$, observe that if the outside sum had bounds running from 0 to $k$, then the sum would equal zero unless $p=-1$, in which case it would equal $(-1)^k$. This observation is \cref{lemma:identity3}, and using it gives
    \[B_i(k) = (-1)^{k}\mathds{1}_{p = -1} - \sum_{r'=0}^{k-1} (-1)^{r'} \binom{p+k}{k-1-r'} \binom{p+r'}{k}\]
    by using the negation of the $t=0$ summand. Furthermore, the summation in the above expression simplifies to
    \[(-1)^{k+p} \mathds{1}_{p>0} - (-1)^{k+p} \mathds{1}_{p<0};\]
    this fact is \cref{lemma:identity4}. Thus
    \begin{align*}
        B_i(k) &= (-1)^{k}\mathds{1}_{p = -1} - (-1)^{k+p} \mathds{1}_{p>0} + (-1)^{k+p} \mathds{1}_{p<0} \\
        &= (-1)^{k+p+1}(\mathds{1}_{p>0} - \mathds{1}_{p<-1})
    \end{align*}
    and the partial fraction decomposition of $B(n, k)$ is
    \begin{align*}
        B(n, k) &= 2k\sum_{p=-k}^{k-1}\frac{1}{n-p} \binom{2k-1}{k+p} \left(\mathds{1}_{p>0} - \mathds{1}_{p<-1}\right) \\
        &= 2k \sum_{i=1}^{k-1} \binom{2k-1}{k+i} \left(\frac{1}{n-i} - \frac{1}{n+i+1}\right),
    \end{align*}
    Combining everything shows that the final probability is
    \begin{align*}
        &\phantom{{}={}} \frac{A(n, k) - (-1)^{n}B(n, k)}{k \binom{2k}{k}} \\
        &= \frac{1}{k\binom{2k}{k}}\left(\binom{2k}{k} + 2k \sum_{i=1}^{k-1} \left((-1)^i - (-1)^n\right) \binom{2k-1}{k+i} \left(\frac{1}{n-i} - \frac{1}{n+i+1}\right)\right) \\
        &= \frac{1}{k} + \frac{4 (-1)^n}{ \binom{2k}{k}} \sum_{\substack{1 \leq i \leq k-1 \\ i \not\equiv n \bmod 2}} \binom{2k-1}{k+i} \left(\frac{1}{n+i+1} - \frac{1}{n-i}\right)
    \end{align*}
    as desired.
\end{proof}

It remains to prove the four identities used in the proof above, which are:

\begin{lemma}\label{lemma:identity1}
    Let $t \leq k \leq n$ be positive integers. Then
    \[\sum_{r=1-k}^{n+1} (-1)^r \binom{r}{t}\binom{n+k}{r+k-1}\binom{n-r}{k-t} = 0.\]
\end{lemma}

\begin{lemma}\label{lemma:identity2}
    Let $p$ be an integer. Then 
    \[\sum_{t=0}^{\infty} (-1)^t \binom{p+t-1}{t} \binom{p+1}{t}\]
    is $1$ if $p \in \{-1, 0\}$ and zero otherwise.
\end{lemma}

\begin{lemma}\label{lemma:identity3}
    Let $p$ and $k$ be integers such that $-k \leq p \leq k-1$. Then
    \[\sum_{t=0}^k (-1)^t \binom{p+t-1}{t} \sum_{r'=0}^{k-1} (-1)^{r'}\binom{r'+t-1}{t}\binom{p+k}{k-1-r'}\binom{p+r'}{k-t}\]
    is $(-1)^k$ if $p = -1$ and zero otherwise.
\end{lemma}

\begin{lemma}\label{lemma:identity4}
    Let $p$ and $k$ be integers such that $-k \leq p \leq k-1$. Then
    \[\sum_{r'=0}^{k-1} (-1)^{r'} \binom{p+k}{k-1-r'} \binom{p+r'}{k}\]
    is $(-1)^{k+p}$ if $p$ is positive, zero if $p$ is zero, and $(-1)^{k+p+1}$ if $p$ is negative.
\end{lemma}

The proofs of these four lemmas are similar. To prove these lemmas, first define $[x^k]f(x)$ to be the coefficient of $x^k$ in $f(x)$'s Laurent expansion, where $k$ is an integer and $f$ is a rational function. The four proofs rely on the following fact:

\begin{fact}\label{fact:binomial}
    Let $n$ and $k$ be integers. Then
    \[[x^k] (1+x)^n = \begin{cases}\binom{n}{k} & k \geq 0 \\ 0 & k < 0.
    \end{cases}\]
\end{fact}

\begin{proof}
    This follows from the generalized binomial thoerem.
\end{proof}

To prove Lemmas \ref{lemma:identity1} to \ref{lemma:identity4}, the general method of attack is as follows:
\begin{itemize}
    \item Rewrite the binomials and adjust the bounds of summation so that one binomial has the variable of summation on the bottom.
    \item Use \cref{fact:binomial} to write all other binomials as Taylor coefficients.
    \item Rearrange the expression such that the coefficient extractions are done last.
    \item Use the binomial theorem to simplify the expression.
    \item Make an observation about the resulting Taylor coefficient.
\end{itemize}
In many cases, this outline must be modified. For example, the first step in the method would require two binomials for double summations: one for each variable of summation. Additionally, one must be careful when using \cref{fact:binomial}, since the Taylor coefficient conversion is only valid when the bottom number in the binomial is nonnegative. Lastly, it could be difficult to extract the final Taylor coefficient. In this case, it sometimes works to replace some $\binom{n}{k}$'s with $\binom{n}{n-k}$'s and retry the method.

When working with infinite sums, it is important to make sure the exponents on negative terms in expansions are bounded. Otherwise, one can derive an apparent contradiction from taking, for example, the constant coefficient on each side of
\[1 + x^{-1} + x^{-2} + \cdots = \frac{1}{1-\frac{1}{x}} -\frac{x}{1-x} = -x - x^2 - x^3 - \cdots.\]
With these heuristics in mind, here are the proofs of the four lemmas.

\begin{proof}[Proof of \cref{lemma:identity1}]
    By adjusting the bounds of summation, it suffices to show that
    \[\sum_{r'=0}^{\infty} (-1)^{r'} \binom{k+n}{r'}\binom{r'-k+1}{t}\binom{k+n-r'-1}{k-t} = 0.\]
    Indeed, by \cref{fact:binomial} the left side equals
    \begin{align*}
        &\phantom{{}={}}\sum_{r'=0}^{\infty} (-1)^{r'} \binom{k+n}{r'} [x^t] (1+x)^{r'-k+1} [y^{k-t}] (1+y)^{k+n-r'-1} \\
        &= [x^t][y^{k-t}](1+x)^{-k+1} (1+y)^{k+n-1}\sum_{r'=0}^\infty \binom{k+n}{r'}(-1)^{r'} (1+x)^{r'}(1+y)^{-r'} \\
        &=  [x^t][y^{k-t}](1+x)^{-k+1} (1+y)^{k+n-1}\left(1-\frac{1+x}{1+y}\right)^{k+n} \\
        &= [x^t][y^{k-t}](1+x)^{-k+1} (1+y)^{k+n-1}(y-x)^{k+n} (1+y)^{-(k+n)}.
    \end{align*}
    Due to the $(y-x)^{k+n}$ term, every term in the Taylor expansion must have degree at least $k+n$. Thus, the $x^t y^{k-t}$ coefficient must be zero.
\end{proof}

\begin{proof}[Proof of \cref{lemma:identity2}]
    There are two cases: $p \geq 1$ and $p \leq 0$. If $p \geq 1$, then using \cref{fact:binomial} gives
    \begin{align*}
        \sum_{t=0}^{\infty} (-1)^t \binom{p+t-1}{t} \binom{p+1}{t} &= \sum_{t=0}^{p+1} (-1)^t [x^0]\frac{(1+x)^{p+t-1}}{x^t} \binom{p+1}{t} \\
        &= [x^0] (1+x)^{p-1} \sum_{t=0}^{p+1} \binom{p+1}{t}(-1)^t \frac{(1+x)^t}{x^t} \\
        &= [x^0] (1+x)^{p-1} \left(1 - \frac{1+x}{x}\right)^{p+1} \\
        &= (-1)^{p+1}[x^{p+1}] (1+x)^{p-1}
    \end{align*}
    which is zero for degree reasons. If $p \leq 0$, then rewriting the first binomial coefficient and applying \cref{fact:binomial} gives
    \begin{align*}
        \sum_{t=0}^\infty (-1)^t \binom{p+t-1}{t} \binom{p+1}{t} &= \sum_{t=0}^{-p} \binom{-p}{t} \binom{p+1}{t} \\
        &= \sum_{t=0}^{-p} \binom{-p}{t} [x^0] \frac{(1+x)^{p+1}}{x^t} \\
        &= [x^0](1+x)^{p+1} \sum_{t=0}^{-p} \binom{-p}{t}\frac{1}{x^t} \\
        &= [x^0](1+x)^{p+1} \left(1+\frac{1}{x}\right)^{-p} \\
        &= [x^{-p}] (1+x),
    \end{align*}
    which is zero when $p \leq -2$ and 1 when $p \in \{-1, 0\}$.
\end{proof}

\begin{proof}[Proof of \cref{lemma:identity3}]
    This proof is only sketched because the details are messy and the ideas have already been written out in the proofs of the other lemmas.
    
    By flipping the direction of summation and swapping the order, the sum becomes
    \[-\sum_{r''=0}^{k-1}(-1)^{r''}\binom{p+k}{r''}\sum_{t'=0}^k (-1)^{t'} \binom{p+k-1-r''}{t'}\binom{p+k-t'-1}{k-t'}  \binom{2k-2-r''-t'}{k-2-r''}\]
    after simplification.
    Since \cref{fact:binomial} and the bounds $0 \leq r'' \leq k-1$, $0 \leq t' \leq k$ give
    \[[y^{k-2}] y^{r''} (1+y)^{2k-2-r''-t'} = \begin{cases}
        \binom{2k-2 - r'' - t'}{k-2-r''} & (r'', t') \neq (k-1, k) \\
        0 & (r'', t') = (k-1, k),
    \end{cases}\]
    the desired sum can be expressed as a sum of the $(r'', t') = (k-1, k)$ summand and a sum over Taylor coefficients. The $(r'', t') = (k-1, k)$ summand equals
    \[\binom{p+k}{k-1}\binom{p}{k} = (-1)^{k}\mathds{1}_{p=-1},\]
    so the entire sum simplifies to
    \[(-1)^{k}\mathds{1}_{p=-1} -[x^k][y^{k-2}] (1+x)^{p+k-1} (1+y)^{2k-2} P_{p,k}(x, y),\]
    where $P_{p,k}(x, y)$ is the series
    \[\sum_{r''=0}^{\infty} (-1)^{r''} \binom{p+k}{r''} y^{r''} (1+y)^{-r''}\sum_{t'=0}^{\infty} (-1)^{t'} \binom{p+k-1-r''}{t'} x^{t'}(1+x)^{-t'} (1+y)^{-t'}.\]
    Simplfying this expression by using the binomial theorem twice gives
    \[P_{p,k}(x, y) = \frac{1}{(1+y+xy)(1+x)^{p+k-1}(1+y)^{p+k-1}}\]
    so the final sum becomes
    \[(-1)^{k}\mathds{1}_{p=-1}-[x^k][y^{k-2}] \frac{(1+y)^{k-1-p}}{1+y+xy}.\]
    Finally, \[[x^a y^b]\frac{(1+y)^{k-1-p}}{1+y+xy}\] is zero when $a > b$, since each $x$ comes paired with a $y$. Thus, the final sum equals $(-1)^{k}\mathds{1}_{p=-1}$, as desired.
\end{proof}

\begin{proof}[Proof of \cref{lemma:identity4}]
    By flipping the direction of summation, the sum is equal to
    \[\sum_{r''=0}^{k-1} (-1)^{k-1-r''} \binom{p+k}{r''} \binom{p+k-1-r''}{k}.\]
    Now, if the summation ranged from 0 to $p+k$, the sum's value would
    \begin{itemize}
        \item not change if $p+k \leq k-1$, and
        \item increase by the value of the summand at $r''=p+k$ if $p+k > k-1$.
    \end{itemize}
    The value of the summand when $r''=p+k$ is
    \[(-1)^{k-1-(p+k)}\binom{p+k}{p+k}\binom{(p+k-1)-(p+k)}{k} = -(-1)^{k+p}.\]
    Thus the final sum, by \cref{fact:binomial} and the binomial theorem, equals
    \begin{align*}
        &\phantom{{}={}} (-1)^{k+p} \mathds{1}_{p \geq 0} + \sum_{r''=0}^{p+k} (-1)^{k-1-r''} \binom{p+k}{r''} \binom{p+k-1-r''}{k}
        \\
        &= (-1)^{k+p}\mathds{1}_{p \geq 0}-(-1)^{k}\sum_{r''=0}^{p+k} (-1)^{r''} \binom{p+k}{r''} [x^k](1+x)^{p+k-1-r''} \\
        &= (-1)^{k+p}\mathds{1}_{p \geq 0} - (-1)^k [x^k] (1+x)^{p+k-1}\sum_{r''=0}^{p+k}\binom{p+k}{r''} (-1)^{r''}(1+x)^{-r''} \\
        &= (-1)^{k+p}\mathds{1}_{p \geq 0} - (-1)^k [x^k] (1+x)^{p+k-1} \left(1 - \frac{1}{1+x}\right)^{p+k} \\
        &= (-1)^{k+p}\mathds{1}_{p \geq 0} - (-1)^k [x^{-p}] (1+x)^{-1} \\
        &= (-1)^{k+p}\mathds{1}_{p \geq 0} - (-1)^{k+p} \mathds{1}_{p \leq 0},
    \end{align*}
    as desired.
\end{proof}

\section{Future directions}\label{section:future}

This paper studies coalescence probabilities, and \cite{BDMS14} studied separation probabilities. The natural generalization is, given a partition $\lambda$ of $k$, what is the probability that the elements $\{1, 2, \dots, k\}$ are distributed according to $\lambda$ among the cycles of a product of two random $n$-cycles? Coalescence probabilities correspond to the partition $\lambda = (k)$, and separation probabilities correspond to the partition $\lambda = (1, 1, \dots, 1)$. Stanley \cite[p.~49]{Sta10} claimed that for sufficiently large $n$, this probability is a rational function in $n$ for each possible parity of $n$. However, he did not give a formula for this probability. Perhaps calculating the partial fraction decomposition of these rational functions could lead to new insights into the behavior of these probabilities.

\section{Acknowledgments}
This work was supported by MIT's UROP+ program under the guidance of mentor Oriol Sol\'e Pi. This research topic was inspired by an MIT class project on randomly gluing the edges of a $2n$-gon to form a closed orientable surface. Hence, I would also like to thank 18.821 professor Roman Bezrukavnikov, writing advisor Susan Ruff, and collaborators Kenta Suzuki and Joseph Camacho for the thought-provoking discussions that led me to choose this topic for my summer research.

\clearpage

\bibliographystyle{amsplain}
\bibliography{references}

\end{document}